\documentclass{amsart}
\usepackage{amssymb, amsthm, amsmath, amsfonts,amscd,bm,fancyhdr}
\usepackage{graphics}
\usepackage{hyperref}
\usepackage[all]{xy}
\usepackage{enumerate}
\usepackage[mathscr]{eucal}
\usepackage{bbm}
\usepackage{tikz}
\usepackage{parskip}
\providecommand{\U}[1]{\protect\rule{.1in}{.1in}}
\def\theenumi{\arabic{enumi}}

\def\theenumii{\alph{enumii}}
\def\p@enumii{\theenumi.}

\def\theenumiii{\arabic{enumiii}}
\def\p@enumiii{(\theenumi)(\theenumii)}

\def\p@enumiv{\p@enumiii.\theenumiii}

\theoremstyle{plain}
\newtheorem{theorem}{Theorem}[section]

\newtheorem{lemma}[theorem]{Lemma}

\newtheorem{proposition}[theorem]{Proposition}

\newtheorem{corollary}[theorem]{Corollary}

\numberwithin{equation}{section}

\theoremstyle{definition}

\newtheorem{definition}[theorem]{Definition}
\newtheorem{example}[theorem]{Example}

\newtheorem{remark}[theorem]{Remark}

\newtheorem{thmab}{Theorem}
\newtheorem*{thm}{Theorem}

\DeclareMathOperator{\FI}{FI}

\DeclareMathOperator{\Tor}{Tor}

\newcommand{\Sn}{\mathfrak{S}}

\newcommand{\Ob}{\mathcal{O}}

\newcommand{\Z}{{\mathbb{Z}}}

\newcommand{\F}{\mathcal{F}}
\newcommand{\mi}{\mathfrak{m}}

\newcommand{\dt}{\bullet}

\newcommand{\arXiv}[1]{\href{http://arxiv.org/abs/#1}{\nolinkurl{arXiv:#1}}}
\newcommand{\arXivV}[2]{\href{http://arxiv.org/abs/#1}{\nolinkurl{arXiv:#1v#2}}}

\title{An application of the theory of $\FI$-algebras to graph configuration spaces}

\author[E.~Ramos]{Eric Ramos}
\address{University of Michigan Department of Mathematics, 530 Church St., Ann Arbor, MI 48109}
\email{egramos@umich.edu}

\thanks{The first author was supported by NSF grant DMS-1704811.}

\begin{document}

\begin{abstract}
Recent work of An, Drummond-Cole, and Knudsen \cite{ADK}, as well as the author \cite{R}, has shown that the homology groups of configuration spaces of graphs can be equipped with the structure of a finitely generated graded module over a polynomial ring. In this work we study this module structure in certain families of graphs using the language of $\FI$-algebras recently explored by Nagel and R\"omer \cite{NR}. As an application we prove that the syzygies of the modules in these families exhibit a range of stable behaviors.
\end{abstract}


\maketitle

\section{Introduction}

There has been a recent raising interest in studying \emph{configuration spaces of graphs} due to their connections with robotics and topological motion planning \cite{Fa,G}. In this paper, a \emph{graph} is a 1-dimensional, compact simplicial complex. Its \emph{$n$-stranded configuration space} is the topological space
\[
U\F_n(G) :=  \{(x_1,\ldots,x_n) \in G^n \mid x_i \neq x_j\}/\Sn_n.
\]

In this paper, we explore the homology groups of these configuration spaces from a relatively new perspective, following the work of An, Drummond-Cole, and Knudsen \cite{ADK}, as well as the author \cite{R}.

For each $q \geq 0$, and each graph $G$, the \emph{total $q$-th homology group} is the graded abelian group
\[
\mathcal{H}_q(G) := \bigoplus_{n \geq 0} H_q(U\F_n(G))
\]
Writing $A_G$ for the integral polynomial ring with variables indexed by the edges of $G$, the following theorem was proven in the case where $G$ is a tree by the author \cite{R}, and in the general case by An, Drummond-Cole, and Knudsen.

\begin{thm}[Ramos, \cite{R}; An, Drummond-Cole, Knudsen, \cite{ADK}]
The graded abelian group $\mathcal{H}_q(G)$ can be equipped with the structure of a finitely generated, graded module over $A_G$.
\end{thm}

It therefore becomes natural to ask whether one can use techniques from commutative algebra to deduce facts about these homology groups. In \cite{R}, the author showed that the total homology groups could be decomposed into sums of graded shifts of square-free monomial ideals, in the cases wherein $G$ is a tree. Other than that work, however, very little has been done in this direction. In this work, we study the modules $\mathcal{H}_q(G)$ for specific families of graphs $G$, using recent work of White and the author \cite{RW}, and L\"utgehetmann \cite{L}, as well as Nagel and R\"omer \cite{NR}.

In \cite{RW}, White and the author introduced the notion of a \emph{finitely generated $\FI$-graph} (see Definition \ref{figraphdef}). For the purposes of this introduction, we restrict to a specific natural class of examples of this structure. Fix two graphs, $G, H$, as well as a choice of vertex $v_G,v_H$ in each. Then for each $n \geq 0$ we may define
\[
G_n := G \bigvee^n H.
\]
The sequence of graphs $\{G_n\}$ is notable for us, as it carries an action by the category $\FI$ of finite sets and injections. Indeed, if we set $[n] := \{1,\ldots, n\}$ and let $f:[n] \hookrightarrow [r]$ be an injection, then we obtain an injective homomorphism of graphs (see Definition \ref{graphdef})
\[
G(f):G_n \rightarrow G_r
\]
which fixes the vertices of $G$, while permuting the vertices of $H$ according to $f$. This structure was first exploited in the study of configuration spaces by L\"utgehetmann in \cite{L}. In particular, the following is proven.

\begin{thm}[L\"utgehetmann, \cite{L}]
Let $G_n$ be as above. Then for each $q,m \geq 0$, the functor
\[
[n] \mapsto H_q(U\F_m(G_n))
\]
from $\FI$ to abelian groups is finitely generated in the sense of \cite{CEF}.
\end{thm}

Therefore, if we fix the number of points being configured, but allow our graph to grow in a certain way, one obtains a finiteness result. On the other hand, we have also seen that if we fix the graph, and allow the number of points being configured to increase, we obtain a similar finiteness result. The main theorem of this work is that these two types of finiteness are compatible with one another in the appropriate sense.

\begin{thmab}\label{mainthmab}
Let $G_n$ be as above. Then for each $q \geq 0$, the functor
\[
[n] \mapsto \mathcal{H}_q(G_n)
\]
from $\FI$ to graded abelian groups is finitely generated as a graded module over the functor
\[
[n] \mapsto A_{G_n}
\]
from $\FI$ to $\Z$-algebras, in the sense of \cite{NR} (see also Definition \ref{fgdef}).
\end{thmab}

\begin{remark}
As previously stated, the most general theorem involves studying a larger class of families of graphs (see Theorem \ref{mainthm}). That being said, it is unclear whether the scope this theorem extends to all finitely generated $\FI$-graphs (see Definition \ref{figraphdef}). This difficulty is fundamentally related to the fact that the $\FI$-algebra
\[
[n] \mapsto A_{G_n}
\]
is not necessarily Noetherian in general. It is an interesting question to ask whether the conclusion of Theorem \ref{mainthmab} will hold for all $\FI$-graphs, even in the case where $G_n = K_n$ is the complete graph on $n$ vertices.
\end{remark}

Coming back to the goal of studying the commutative algebra of $\mathcal{H}_q(G)$, we note that Theorem \ref{mainthmab} has concrete consequences related to the Betti numbers of these modules. If $M$ is a finitely generated graded module over the polynomial ring $k[x_1,\ldots,x_n]$, where $k$ is a field, then it admits a minimal free resolution
\[
0 \rightarrow F^{(n)} \rightarrow \ldots \rightarrow F^{(0)} \rightarrow M \rightarrow 0
\]
Tensoring with the residue field $k$, we define the \emph{$j$-th Betti number in homological degree $p$} to be
\[
\beta_{p,j}(M) := \dim_k(\Tor_p(M,k)_j).
\]
Knowledge of the Betti numbers allows one to access the number of generators in each grade of the $p$-th syzygies. Our second main result is that the Betti numbers of the modules $\mathcal{H}_q(G_n)$ stabilize in a strong sense.

\begin{thmab}\label{bettithm}
Let $G_n$ be as above, and let $k$ be a field. Then for each $p,j \geq 0$, the function
\[
n \mapsto \beta_{p,j}(\mathcal{H}_q(G_n;k))
\]
agrees with a polynomial for $n \gg 0$. Moreover, there exists a finite list of integers
\[
j_0(\mathcal{H}_q(G_n;k),p) < \ldots < j_t(\mathcal{H}_q(G_n;k),p)
\]
such that for all $n \gg 0$
\[
\beta_{n,p,j}(\mathcal{H}_q(G_n;k)) \neq 0 \iff j \in \{j_0(\mathcal{H}_q(G_n;k),p), \ldots, j_t(\mathcal{H}_q(G_n;k),p)\}.
\]
\end{thmab}

\begin{remark}
It has been known since at least \cite[Theorem 3.16]{KP} that the $A_G$-module $\mathcal{H}_1(G)$ is generated in degree $\leq 2$, for all graphs $G$. This result has also appeared in An, Drummond-Cole, and Knudsen \cite[Proposition 5.6]{ADK}. To the knowledge of the author, the present work is the first to provide some evidence of such uniform boundedness for the higher syzygies of $\mathcal{H}_q$, for $q \geq 2$.
\end{remark}

\section{Graph configuration spaces}

\begin{definition}\label{graphdef}
A \textbf{graph} is a 1-dimensional compact simplicial complex. We call the 0-simplicies of $G$ its \textbf{vertices}, $V(G)$, while the 1-simplices are its \textbf{edges}, $E(G)$. If $v$ is a vertex of $G$, then the \textbf{degree} of $v$, denoted $\mu(v)$, is the number of edges adjacent to $v$. A vertex is said to be \textbf{essential} if it has $\mu(v) \neq 2$. A \textbf{homomorphism of graphs} is a map between vertices
\[
\phi:V(G) \rightarrow V(G')
\]
which preserves adjacency.

If $G$ is a graph, the \textbf{(unordered) $n$-strand configuration space of $G$}, denoted $U\F_n(G)$, is given by
\[
U\F_n(G) := \{(x_1,\ldots,x_n) \in G^n \mid x_i \neq x_j\}/\Sn_n
\]
where $\Sn_n$ acts by permuting coordinates in the natural way.
\end{definition}

Our main concern will be developing a better understanding of the homology groups $H_q(U\F_n(G))$. While much of the current literature on the subject has been concerned with their relation to $\pi_1(U\F_n(G))$ (see \cite{FS,KKP}, for instance), in this work we will approach their study from a more recent algebraic perspective.

\begin{definition}
Writing $E(G)$ for the set of edges of $G$, We will use $A_G$ to denote the polynomial ring
\[
A_G := \Z[x_e \mid e \in E(G)]
\]
\end{definition}

In \cite{R}, the author showed that the total $q$-homology group, $\mathcal{H}_q(G) = \bigoplus_n H_q(U\F_n(G))$ could be endowed with the structure of a finitely generated, graded module over $A_G$ whenever $G$ was a tree. Following this, An, Drummond-Cole, and Knudsen showed that this could be extended to all graphs \cite{ADK}. Our interest in this paper will be to understand the syzygies of the module $\mathcal{H}_q$ for a certain family of graphs, which we call edge-linear (see Definition \ref{edgelin}). In particular, we prove facts about the Betti numbers, in the commutative algebraic sense, by applying the techniques of $\FI$-algebras recently developed by Nagel and R\"omer \cite{NR}, building on the work of Sam \cite{Sa,Sa2}, Snowden \cite{Sn}, and others.

\begin{definition}
Let $G$ be a graph. A \textbf{half-edge} of $G$ is a pair $h = (v,e)$ where $v$ is a vertex and $e$ is an edge adjacent to $v$. We will use $v(h)$ and $e(h)$ to denote the vertex and edge associated to $h$, respectively. The set of half edges with $v(h) = v$ will be written $H(v)$.

For a vertex $v \in V(G)$, we set $Sw(v)$ to be the abelian group freely generated by the symbols
\[
Sw(v) = <\emptyset, v, h \in H(v)>.
\]
We consider $Sw(v)$ to be a bigraded abelian group by setting
\[
|v| = (0,1), |\emptyset| = (0,0), |h| = (1,1).
\]
The \textbf{Swiatkowski complex} associated to $G$ is the differential bigraded $A_G$-module
\[
Sw(G) = A_G \otimes \bigotimes_v Sw(v),
\]
with
\[
|e| = (0,1)
\]
and differential defined by
\[
\partial(h) = e(h) - v(h), \partial(e) = \partial(v) = \partial(\emptyset) = 0.
\]

Expressed more classically, we may think of the Swiatkowski complex in the following way. For each $q \geq 0$, write $Sw_q(G)$ to be the subgroup of $Sw(G)$ generated by pure tensors of grade $(q,x)$ for some $x \geq 0$. Then $Sw_q(G)$ can be thought of as a graded module over $A_G$. The Swiatkowski complex is the complex of graded $A_G$-modules
\[
Sw_{\star}(G) : \ldots \rightarrow Sw_q(G) \stackrel{\partial}\rightarrow Sw_{q-1}(G) \rightarrow \ldots \rightarrow Sw_0(G) \rightarrow 0
\]
where $\partial$ is as above. We will use these two perspectives interchangeably in what follows.
\end{definition}

\begin{remark}
Because it will be relevant later, we note that the assignment $G \mapsto Sw(G)$ is functoral. In particular, given an injective homomorphism of graphs $\phi:G \rightarrow G'$ one obtains a morphism $Sw(G) \rightarrow Sw(G)$. The reader can check that this morphism preserves all the relevent structure of the complex including the bigrading, as well as the differential.

We will follow the notation of An, Drummond-Cole, and Knudsen regarding shifts in the grading of bigraded modules. If $M$ is a differential bigraded module over some algebra $A$, then we we will write $M\{1\}$ to denote the bigraded module with
\[
M\{1\}_{(a,b)} = M_{(a,b-1)}.
\]
Similarly, we will write $M[1]$ to denote the module with bigrading
\[
M[1]_{(a,b)} =M_{(a-1,b)}.
\]
The differentials of these modules follow the usual sign conventions for degree shifting.
\end{remark}

The Swiatkowski complex was orginally expressed in the above form by An, Drummond-Cole, and Knudsen in \cite{ADK}, though they credit work of Swiatkowski \cite{Sw} for the origin of the idea. The following theorem is extremely relevant for the remainder of this paper.

\begin{theorem}[An, Drummond-Cole, and Knudsen, Theorem 4.3 \cite{ADK}]\label{swiahomology}
There is an isomorphism of functors from the category of graphs and injective homomorphisms to bigraded abelian groups,
\[
\mathcal{H}_\star(G) \cong H_\star(Sw(G))
\]
In particular, for any graph $G$ the total $q$-homology group $\mathcal{H}_q(G)$ carries the structure of a finitely generated graded module over the polynomial ring $A_G$.
\end{theorem}

In the case where $G$ is a tree, the structure of $\mathcal{H}_q(G)$ as a $A_G$-module was explored by the author \cite{R}. It is shown in that work that $\mathcal{H}_q(G)$ decomposes as a direct sum of graded shifts of square-free monomial ideals. In the work \cite{ADK}, An, Drummond-Cole and Knudsen also explore $\mathcal{H}_q(G)$ in the case where $G$ has maximum valency 3.

For the purpose of doing computations later in this work, we will need to present a simplification of the Swiatkowski complex.

\begin{definition}
Let $G$ be a graph with no isolated vertices. For each vertex $v$ of $G$ we may enumerate the half-edges containing $v$ by $\{h_i\}$. Define a bigraded subgroup of $S(v)$ by,
\[
\widetilde{S}(v) := <\emptyset, h_{i,j}>_{i,j},
\]
where $h_{i,j} = h_i - h_j$. Then the \textbf{reduced Swiatkowski complex} is the submodule of $Sw(G)$ defined by
\[
\widetilde{Sw}(G) := A_G \otimes \bigotimes_{v} \widetilde{S}(v)
\]
\end{definition}

One observes that for any pair of half edges $h_i,h_j$ containing some fixed vertex $v$, one has
\[
\partial(h_{i,j}) = e(h_i) - e(h_j)
\]
It follows that $\widetilde{Sw}(G)$ is indeed a differential bigraded submodule of the Swiatkowski complex. 

\begin{proposition}[An, Drummond-Cole, Knudsen, Proposition 4.9 \cite{ADK}]
If $G$ does not contain any isolated vertices, then the inclusion
\[
\widetilde{Sw}(G) \hookrightarrow Sw(G)
\]
is a quasi-isomorphism.
\end{proposition}

\section{$\FI$-algebras and edge-linear $\FI$-graphs}

\begin{definition}
For each integer $n \geq 0$, we will write $[n]$ to denote the finite set $[n] := \{1,\ldots,n\}$. The category $\FI$ is that whose objects are the sets $[n]$ and whose morphisms are set injections.

If $k$ is a commutative Noetherian ring, then an \textbf{$\FI$-module over $k$} is a functor $M$ from $\FI$ to the category of $k$-modules. We will often write $M_n$ to denote $M([n])$.
\end{definition}

The representation theory of $\FI$ has recently risen to popularity due in large part to the seminal works of authors such as Church, Ellenberg, Farb, and Nagpal \cite{CEF,CEFN}, Djament and Vespa \cite{D,DV}, Sam and Snowden \cite{SS,SS2}, and many others. In this work we will be primarily concerned with two constructions stemming from $\FI$: $\FI$-algebras and $\FI$-graphs.

\begin{definition}\label{fgdef}
If $k$ is a commutative Noetherian ring, an $\FI$-algebra is a functor $\mathbf{A}$ from $\FI$ to the category of $k$-algebras. Given an $\FI$-algebra $\mathbf{A}$, an \textbf{$\mathbf{A}$-module} is an $\FI$-module $M$ satisfying the following:
\begin{enumerate}
\item $M_n$ is an $\mathbf{A}_n$-module for each $n$;
\item if $f:[n] \hookrightarrow [m]$ is an injection of sets then the following diagram commutes for all $a \in A_n$
\[
\begin{CD}
M_n @> M(f)>> M_m\\
@V a VV      @V A(f)(a) VV\\
M_n @> M(f) >> M_m
\end{CD}
\]
\end{enumerate}
An $\mathbf{A}$-module $M$ is said to be \textbf{finitely generated} so long as there is a finite collection of elements $\{v\} \subseteq \bigoplus M_n$ which no proper sub $\mathbf{A}$-module contains. We say that a finitely generated $\mathbf{A}$-module $M$ is \textbf{Noetherian} if every sub $\mathbf{A}$-module of $M$ is also finitely generated. We say that $\mathbf{A}$ is itself \textbf{Noetherian} if all finitely generated modules over $\mathbf{A}$ are Noetherian.
\end{definition}

\begin{example}
The simplest example of an $\FI$-algebra is that whose points are defined by $\mathbf{A}_n = k$, and whose transition maps are all the identity. In this case $\mathbf{A}$-modules are no different than $\FI$-modules over $k$.

The second example of an $\FI$-algebra, and the example which will be the fundamental object of study in this paper, is that given on points by $\mathbf{A}_n = k[x_1,\ldots,x_n]$, and whose transition maps act on indices in the obvious way.

Note that for any fixed $d \geq 0$, and any $\FI$-algebra $\mathbf{A}$, we may always define a new algebra given by the tensor power
\[
(\mathbf{A}^{\otimes d})_n := (\mathbf{A}_n)^{\otimes d} 
\]
\end{example}

It is famously known that all finitely generated $\FI$-modules over Noetherian rings are themselves Noetherian (see \cite{CEFN} for one such proof). However, finitely generated modules over a general $\FI$-algebra do not share this property. This is true even if we assume that $\mathbf{A}$ is itself finitely generated. Despite this, one can prove the following.

\begin{theorem}[Nagel and R\"omer, Corollary 7.9 \cite{NR}]\label{noeth}
If $k$ a commutative Noetherian ring, the $\FI$-algebra $\mathbf{A}$ given by $\mathbf{A}_n = k[x_1,\ldots,x_n]$ is Noetherian. More generally, for any fixed integer $d \geq 0$, $\mathbf{A}^{\otimes d}$ is Noetherian.
\end{theorem}

\begin{remark}
It is still open whether the Noetherian property of $\FI$-algebras is preserved by tensor products. The above special case was proven by Nagel and R\"omer in the provided source.
\end{remark}

Theorem \ref{noeth} has a variety of interesting consequences in the case where we keep track of the natural grading on $\mathbf{A}_n$.

\begin{definition}
We say that an $\FI$-algebra $\mathbf{A}$ over a field $k$ is \textbf{standard graded} if every algebra $\mathbf{A}_n$ is a graded $k$-algebra generated in degree 1 with $\mathbf{A}_{n,0} = k$, and every induced map $A(f)$ is a morphism of graded $k$-algebras. We say that an $\mathbf{A}$-module $M$ is \textbf{graded} if each $M_n$ is a graded module over $\mathbf{A}_n$, and each induced map $M(f)$ preserves this grading.

If $M$ is a finitely generated graded $\mathbf{A}$-module, then for each $p,n \geq 0, j \in \Z$ we define $\beta_{n,p,j}(M)$ to be the Betti number
\[
\beta_{n,p,j}(M) := \dim_k (\Tor_p^{\mathbf{A}_n}(M_n,k)_j).
\]
\end{definition}

\begin{corollary}[Nagel and R\"omer, Theorem 7.7 \cite{NR}]\label{betti}
Let $\mathbf{A}$ be a finitely generated standard graded $\FI$-algebra over a field $k$. Then for any integers $p,j \geq 0$, and any finitely generated graded $\mathbf{A}$-module $M$, the function
\[
n \mapsto \beta_{n,p,j}(M)
\]
agrees with a polynomial for $n \gg 0$. Moreover, there exists a finite list of integers
\[
j_0(M,p) < \ldots < j_t(M,p)
\]
such that for all $n \gg 0$
\[
\beta_{n,p,j}(M_n) \neq 0 \iff j \in \{j_0(M,p), \ldots, j_t(M,p)\}.
\]
\end{corollary}

\begin{remark}
The first part of Corollary \ref{betti} does not explicitly appear in \cite{NR}, although it is a clear consequence of the proof of \cite[Theorem 7.7]{NR}.
\end{remark}

One way to think about Corollary \ref{betti} is that if $M$ is a finitely generated graded $\mathbf{A}$-module, then in each fixed degree $p$ the $p$-th syzygies eventually stabilize. One should note, however, that it is not the case that this stabilization is uniform across all $p$. Namely, one should not expect the regularity of $M_n$ to be eventually independent of $n$.

\textbf{For the remainder of the paper, we will reserve $\mathbf{A}$ to mean the $\FI$-algebra defined by $\mathbf{A}_n = \Z[x_1,\ldots,x_n]$.}

The second $\FI$-structure which will be important to us are $\FI$-graphs.

\begin{definition}\label{figraphdef}
An $\FI$-graph is a functor $G_\dt$ from $\FI$ to the category of graphs with graph homomorphisms. As with $\FI$-algebras, we will often write $G_n$ to denote $G_\dt([n])$ and $G(f)$ to denote $G_\dt(f)$. We say that $G_\dt$ is finitely generated if for $n \gg 0$, the vertex set $V(G_{n+1})$ is equal to $\cup_f G(f)(V(G_n))$, where the union is over injections $f:[n] \hookrightarrow [n+1]$.

We say that $G_\dt$ is torsion-free if the induced maps $G(f)$ are injective for all $f$.
\end{definition}

\begin{example}\label{exlin}
The most typical example of an $\FI$-graph is the complete graph $K_n$. This is the graph whose vertices are labeled by $[n]$, and whose edges are labeled by pairs $\{i,j\}$ with $i \neq j$.

For the purposes of this paper a more relevant example is given as follows. Let $(G,v_G)$ and $(H,v_H)$ be any pair of graphs with a choice of vertex. Then we define
\[
G_n := G \bigvee^n H.
\]
For any injection, the induced map on $G_\dt$ is defined by fixing the copy of $G$, and permuting the copies of $H$. For example, if $G$ is a single vertex, and $H$ is an edge, then $G_n$ is the star graph, which has a unique essential vertex of degree $n$. This example was first discussed in the context of $\FI$-graphs by L\"utgehetmann in \cite{L}, although he does not use the language of $\FI$-graphs.
\end{example}

$\FI$-graphs were first introduced by the author and Graham White in \cite{RW}. Configuration spaces of $\FI$-graphs are shown to exhibit a variety of stable behaviors in that work, extending results of L\"utgehetmann from \cite{L}. In this work we combine these stabilization results with the aforementioned $A_G$-action.

\begin{definition}\label{edgelin}
Let $G_\dt$ be a finitely generated $\FI$-graph. Then it is proven in \cite{RW} that the function
\[
n \mapsto |E(G_n)|
\]
agrees with a polynomial for all $n \gg 0$. We say that $G_\dt$ is \textbf{edge-linear} if it is finitely generated, and $|E(G_n)|$ agrees with a linear polynomial for all $n \gg 0$.
\end{definition}

\begin{example}
The second example given in Example \ref{exlin} is edge-linear. Moreover, for any fixed $m \geq 1$ the $\FI$-graph defined by $G_n = K_{n,m}$, the complete bipartite graph on $(n+m)$-vertices, is also edge-linear. 

Both of these examples can be thought of as being construct via the following gluing procedure. Let $G$ and $G'$ be two graphs, each sharing a common subgraph $H$. Then for any $n \geq 0$, we obtain a new graph
\[
G \bigsqcup_{H}^n G'
\]
by gluing $n$ copies of $G'$ to $G$ along $H$. This family of graphs may be given an $\FI$-graph structure, where the symmetric group acts by permuting the copies of $G'$, while fixing $G$. Graphs of this kind were originally studied by L\"utgehetmann in \cite{L}. The Classification Theorem \ref{class} essentially implies that all edge-linear $\FI$-graphs are isomorphic to something of the form $G \bigsqcup_H^\dt G'$, for some $G,G',H$.
\end{example}
 
Edge-linear $\FI$-graphs will be useful in this work due to their connection to the $\FI$-algebra $\mathbf{A}$. We explore this connection as a corollary to the following classification result.

\begin{theorem}\label{class}
Let $G_\dt$ be an edge-linear $\FI$-graph. Then there exists a pair of graphs $G,G'$, with a common subgraph $H$, such that for all $n \gg 0$
\[
G_n \cong G \bigsqcup_H^n G'
\]
\end{theorem}

\begin{proof}
We begin by observe that an edge-linear $\FI$-graph  must also have linear growth in its number of vertices. Indeed, this follows from the graph-theoretic inequality
\[
|V(G)| \leq 2|E(G)|/\delta(G)
\]
where $\delta(G)$ is the minimum degree of a vertex of $G$. Appealing to the classification theorem for $\FI$-sets \cite[Theorem A]{RSW}, we conclude that $V(G_n)$ is isomorphic to a disjoint union of some number of copies of the $\Sn_n$-action on $[n]$, along with some number of invariant vertices.

The only $\Sn_n$-equivariant relation on $[n]$ is the identity relation $\{(i,i) \mid i \in [n]\}$, and the anti-identity relation $\{(i,j) \mid i \neq j\}$. In the latter case, the number of edges does not grow linearly in $n$. Therefore, there are no edges between two vertices in the same $\Sn_n$-orbits. If we write two distinct $\Sn_n$-orbits of vertices as
\[
\Ob_1 = \{(i,1) \mid i \in [n]\}, \Ob_2 = \{(i,2) \mid i \in [n]\}
\]
then the same argument before shows that there cannot be edges between vertices of the form $(i,1)$ and $(j,2)$ if $i \neq j$. Finally, if there is an edge between $(i,1)$ and $(i,2)$ for some $i$, then there must be an edge between $(i,1)$ and $(i,2)$ for all $i$. This then leads into the construction of the desired graphs $G,G',$ and $H$: let $\widetilde{G}$ be the graph whose vertices are labeled by $\Sn_n$-orbits of vertices of $G_n$, and whose edges indicate the existence of an edge in $G_n$ between two elements of the orbits. We let $G$ be the induced subgraph of $\widetilde{G}$ whose vertices correspond to singleton orbits, let $G'$ be any subgraph of $\widetilde{G}$ containing all the non-singleton orbits and all edges adjacent to these orbits, and let $H$ be the intersection in $\widetilde{G}$ of $G$ and $G'$.
\end{proof}

\begin{remark}
It was revealed in the above proof that the choices of $G'$ and $H$ are not necessarily unique, though $G$ is unique.
\end{remark}

The asymptotic nature of Theorem \ref{class} is necessary. For instance, one may define an $\FI$-graph which agrees with the bipartite graph $K_{n,n}$ for $n \leq 10$, and collapses to a single edge for all larger $n$. This observation inspires the following definition.

\begin{definition}
Let $G_\dt$ be an edge-linear $\FI$-graph, and let $m$ be sufficiently large so that $G_n \cong G \bigsqcup_H^n G'$ for all $n \geq m$, and for some $G,G',H$. The \textbf{tail} of $G_\dt$, $G_\dt^{\gg}$, is the edge-linear $\FI$-graph,
\[
G_n^{\gg} := \begin{cases} \emptyset &\text{ if $n < m$}\\ G_n &\text{ otherwise.}\end{cases}
\]

We also define an $\FI$-algebra $\mathbf{A}_{G_{\dt}}$ by setting
\[
\mathbf{A}_{G_{\dt},n} := A_{G \bigsqcup_H^n G'}
\]
for all $n \geq 0$, where $G,G'$, and $H$ are as above.
\end{definition}

\begin{remark}
While the tail of an $\FI$-graph depends on the choice of $m$, this choice will not matter in what follows.
\end{remark}

\begin{proposition}\label{algfg}
Let $G_\dt$ be an edge-linear $\FI$-graph. Then there is some $d \geq 0$ such that $\mathbf{A}_{G_{\dt}}$ is a quotient of 
\[
\mathbf{A}^{\otimes d}.
\]
In particular, $\mathbf{A}_{G_{\dt}}$ is Noetherian.
\end{proposition}

\begin{proof}
Appealing to Theorem \ref{class}, it follows that $E(G_n)$ has $a+b$ $\Sn_n$-orbits: $a$ orbits isomorphic to the action of $\Sn_n$ on $[n]$, and $b$ singleton orbits. We conclude that $\mathbf{A}_{G_{\dt}}$ is isomorphic to a quotient of $\mathbf{A}^{\otimes (a+b)}$. Namely, the quotient which identifies the variables of the first $a$ tensor factors with the elements in the non-singleton orbits of $E(G_n)$, and identities the variables of the  final $b$ tensor factors with the $b$ singleton orbits.

The second part of the proposition follows from the fact that any quotient of a Noetherian $\FI$-algebra must be Noetherian, and Theorem \ref{noeth}.
\end{proof}

The goal of the next section is to show that if $G_\dt$ is an edge-linear $\FI$-graph, then the assignments
\[
n \mapsto \mathcal{H}_q(G_\dt)
\]
defines a finitely generated module over $\mathbf{A}_{G_{\dt}}$. To do this we will make critical use of the fact that $\mathbf{A}_{G_{\dt}}$ is Noetherian.

\section{The proofs of the main theorems}

For the remainder of this section we fix an edge-linear $\FI$-graph $G_\dt$. \textbf{In fact, we substitute $G_\dt$ by its tail $G^\gg_\dt$ so that $G_\dt$ is torsion free}. We will also assume that $G_n$ does not contain any isolated vertices for all $n$. In view of Corollary \ref{betti} and Proposition \ref{algfg}, it will suffice to show that the $\FI$-module
\[
\mathcal{H}_q(G_\dt) := \bigoplus_{n \geq 0} H_q(U\F_n(G_\dt))
\]
is finitely generated and graded over the $\FI$-algebra $\mathbf{A}_{G_\dt}$. We accomplish this by proving that the collection of reduced Swiatkowski complexes $\{\widetilde{Sw}(G_n)\}_{n \geq 0}$ can be extended to a complex of graded $\mathbf{A}_{G_\dt}$-modules, $\widetilde{Sw}(G_\dt)$. We will then show that the terms in the complex $\widetilde{Sw}(G_\dt)$ are finitely generated $\mathbf{A}_{G_\dt}$-modules. The Noetherian property (Proposition \ref{algfg}) will imply that $\mathcal{H}_q(G_\dt)$ is finitely generated, as desired.

\begin{lemma}
For any injection of sets $f:[n] \hookrightarrow [r]$, and any $q \geq 0$ the following diagram commutes
\[
\begin{CD}
Sw_q(G_r) @ > \partial >> Sw_{q-1}(G_r)\\
@A Sw(G(f)) AA            @A Sw(G(f)) AA\\
Sw_q(G_n) @> \partial >> Sw_{q-1}(G_n)
\end{CD}
\]
In particular, the collection of Swiatkowski complexes $\{Sw(G_n)\}_{n \geq 0}$ can be extended to a complex of graded $\mathbf{A}_{G_\dt}$-modules.
\end{lemma}

\begin{proof}
This follows from the definition of $\partial$, as well as the definition of $Sw(G(f))$. Indeed, we have for any half edge $h = (v,e)$
\[
\partial(Sw(G(f))(h)) = \partial(G(f)(e),G(f)(v)) = G(f)(e) - G(f)(v) = Sw(G(f))(\partial(h)).
\]
\end{proof}

One observes that, for any injection $f:[n] \hookrightarrow [r]$, $Sw(G(f))$ restricts to a map $\widetilde{Sw}(G_n) \rightarrow \widetilde{Sw}(G_r)$. The above lemma therefore shows that the reduced Swiatokowski complex can also be extended to a complex of graded $\mathbf{A}_{G_\dt}$-modules.

\begin{lemma}
For each $q\geq 0$, the $\mathbf{A}_{G_\dt}$-module $\widetilde{Sw}_q(G_\dt)$ is finitely generated.
\end{lemma}

\begin{remark}
One should observe that this lemma is false if we insist on working with the usual Swiatkowski complex instead of the reduced one.
\end{remark}

\begin{proof}
We begin by defining an $\FI$-set, in the sense of \cite{RSW}, $X_{q,d}(G_\dt)$ For each $n$, the elements of $X_{q,d}(G_\dt)$ are tuples
\[
(m,\widetilde{h}_1,\ldots,\widetilde{h}_q)
\]
where $m$ is a monomial in $A_{G_n}$ of degree $d$, $\widetilde{h}_i$ is a pair of distinct half-edges containing a common vertex of $G_n$, and $\widetilde{h}_i, \widetilde{h}_j$ are associated to different vertices whenever $i \neq j$. It isn't hard to see that $X_{q,d}(G_\dt)$ is a $\FI$-subset of a product of finitely generated $\FI$-sets. This implies that it is itself finitely generated (see \cite[Proposition 4.2]{RSW}).

Note that we may associated to each element of $X_{q,d}(G_n)$ a generating element of $\widetilde{Sw}_q(G_n)$. In fact, we have that
\[
\widetilde{Sw}(G_\dt) \cong \bigoplus_{d \geq 0} \Z X_{q,d}(G_\dt)
\]
where $\Z X_{q,d}(G_\dt)$ is the $\FI$-module $\Z$-linearization of $X_{q,d}(G_\dt)$. This shows that $\widetilde{Sw}(G_\dt)$ is finitely generated, as desired.
\end{proof}

As previously stated, these two lemmas and the Noetherian property immediately lead to the following consequence.

\begin{theorem}\label{mainthm}
Let $G_\dt$ be an edge-linear $\FI$-graph. For each $q\geq 0$, the $\mathbf{A}_{G_\dt}$-module $\mathcal{H}_q(G_\dt)$ is finitely generated.
\end{theorem}

\begin{remark}
The two proceeding lemma do not use edge-linearity of the $\FI$-graph $G_\dt$. Indeed, the only thing preventing one from concluding an analog of Theorem \ref{mainthm} for all $\FI$-graphs is the failing of Noetherianity outside of the linear case. It is unclear whether one should expect such a finite generation result in these cases.
\end{remark}

Base change from $\Z$ to any field will preserve finite generation of $\FI$-modules, and so we also obtain the following as a corollary to the above

\begin{corollary}
Let $G_\dt$ be an edge-linear $\FI$-graph, and let $k$ be a field. Then for each $p,j \geq 0$, the function
\[
n \mapsto \beta_{p,j}(\mathcal{H}_q(G_n;k))
\]
agrees with a polynomial for $n \gg 0$. Moreover, there exists a finite list of integers
\[
j_0(\mathcal{H}_q(G_n;k),p) < \ldots < j_t(\mathcal{H}_q(G_n;k),p)
\]
such that for all $n \gg 0$
\[
\beta_{n,p,j}(\mathcal{H}_q(G_n;k)) \neq 0 \iff j \in \{j_0(\mathcal{H}_q(G_n;k),p), \ldots, j_t(\mathcal{H}_q(G_n;k),p)\}.
\]
\end{corollary}

\section{An inductive method for computation}

In this section we outline an inductive method with which one can compute the $\mathbf{A}_{G_\dt}$-module $\mathcal{H}_q(G_\dt)$. This method extends a similar method originally described by An, Drummond-Cole, and Knudsen \cite[Proposition 5.13]{ADK}. We then conclude the section by performing some example computations

As with the previous section, we fix a edge-linear $\FI$-graph $G_\dt$, which we make torsion free by working with its tail.

\begin{definition}
The theory of finitely generated $\FI$-graphs tells us that for $n \gg 0$ the $\Sn_n$-orbits of vertices of $G_n$ eventually stabilize (see, \cite[Proposition 3.6]{RSW}), in the sense that the transition maps of $G_\dt$ induce isomorphisms
\[
V(G_n)/\Sn_n \cong V(G_{n+1})/\Sn_{n+1}.
\]
Let $m \gg 0$ be in this stable range, and let $v$ be a vertex of $G_m$ which is fixed by $\Sn_m$. We will often say that $v$ is an \textbf{invariant vertex of $G_\dt$}. Then all images of $v$ under all transition maps of $G_\dt$ are also invariant under the action of the relevant symmetric group. By abuse of notation, we say that $v$ is a vertex of $G_n$ for all $n \geq m$.

The \textbf{blow up of $G_\dt$ at $v$} is an $\FI$-graph $Bl_v(G_\dt)$, defined as follows. For $n \geq m$, $Bl_v(G_n)$ is the graph with vertex set
\[
V(Bl_v(G_n)) := (V(G_n) - \{v\}) \cup (\{v\} \times H(v)),
\]
where $H(v)$ is the set of half-edges containing $v$. The edges of $Bl_v(G_n)$ are the edges of $G_n$, altered in the obvious way to account for the new vertices. We also set $Bl_v(G_n) = \emptyset$ for $n < m$. See \cite[Figure 9]{ADK} for a pictorial representation of the blow up.

Let $f:[n] \hookrightarrow [r]$ be an injection of sets. Then the transition map induced by $f$ is defined as follows. On vertices not of the form $v \times h$, the transition map agrees with $G(f)$. Otherwise, we have
\[
Bl_v(G_\dt)(f)(v \times h) = v \times G(f)(h).
\]
This is well defined, as $v$ was an invariant vertex.
\end{definition}

\begin{remark}
We will often talk of invariant vertices of an $\FI$-graph, without making explicit reference to what degree they originate from. Because much of what we care above is asymptotic, this shouldn't cause confusion.
\end{remark}

Following the techniques of \cite{ADK}, our goal will be to relate $\mathcal{H}_q(G_\dt)$ to $\mathcal{H}_q(Bl_v(G_\dt))$. We first observe that $\mathcal{H}_q(Bl_v(G_\dt))$ is a $\mathbf{A}_{G_\dt}$-module. Indeed, the construction of the blowup does not erase, or meaningfully alter edges.

\begin{lemma}
Let $v$ be an invariant vector of $G_\dt$. Then there is an inclusion of complexes
\[
\widetilde{Sw}(Bl_v(G_\dt)) \hookrightarrow \widetilde{Sw}(G_\dt).
\]
\end{lemma}

\begin{proof}
one can realize $\widetilde{Sw}(Bl_v(G_n))$ as the subcomplex of $\widetilde{Sw}(G_n))$ generated by pure tensors whose $v$-th term is $\emptyset$. The fact that this extends to the desired inclusion follows from the fact that $v$ was chosen to be an invariant vector.
\end{proof}

The above proof immediately points to how one can describe the cokernel of $\widetilde{Sw}(Bl_v(G_\dt)) \hookrightarrow \widetilde{Sw}(G_\dt)$. To accomplish this we will need some notation. If $v$ is an invariant vector of $G_\dt$, then we define an $\FI$-module $\widetilde{H}(v)$ to be that for which
\[
\widetilde{H}(v)_n = \Z<h_i-h_j \mid h_i,h_j \in H(v)_n>,
\]
where $H(v)_n$ is the set of half-edges containing $v$, thought of as a vertex of $G_n$. We now extend the above lemma as follows. Note that this result generalizes a lemma of An, Drummond-Cole, and Knudsen \cite{ADK}.

\begin{lemma}
Let $v$ be an invariant vector of $G_\dt$. Then there is an exact sequence of complexes
\[
0 \rightarrow \widetilde{Sw}(Bl_v(G_\dt)) \rightarrow \widetilde{Sw}(G_\dt) \rightarrow \widetilde{H}(v) \otimes \widetilde{Sw}(Bl_v(G_\dt))\{1\}\rightarrow 0
\]
\end{lemma}

It therefore follows that computing $\mathcal{H}_q(G_\dt)$ can often times be reduced to computing $\mathcal{H}_q(Bl_v(G_\dt))$. This inductive approach becomes even more appealing when one considers how simple its base case is. Indeed, the classification theorem implies that after sufficiently many blow-ups one is left with some fixed number of disjoint edges, invariant under the action of the symmetric group, paired with $n$ disjoint copies of some graph, which the symmetric group permutes.

\begin{example}

Take $G_\dt$ to be the star graph
\[
G_n = K_{n,1}
\]
This graph has a unique invariant vertex, and $Bl_v(G_n)$ is a disjoint union of $n$ edges. It is fact that $U\F_n(G_n)$ only has homology in degrees 0 and 1 (see, for instance, \cite{Sw}). We therefore have,
\[
0 \rightarrow \mathcal{H}_1(G_\dt) \rightarrow S^{(\dt-1,1)} \otimes \mathcal{H}_0(Bl_v(G_\dt))\{1\} \rightarrow \mathcal{H}_0(Bl_v(G_\dt)) \rightarrow \mathcal{H}_0(G_\dt) \rightarrow 0,
\]
where $S^{(\dt-1,1)}$ is the $\FI$-module for which $S^{(n-1,1)}$ is the standard representation of $\Sn_n$. Examining this exact sequence from right to left, we note that $\mathcal{H}_0(G_n) = \Z$ for all $n \geq 0$, while $\mathcal{H}_0(Bl_v(G_\dt)) \cong \mathbf{A}$. The map 
\[
\mathcal{H}_0(Bl_v(G_\dt)) \rightarrow \mathcal{H}_0(G_\dt) \rightarrow 0 
\]
identifies all of the disjoint edges of $Bl_v(G_\dt)$. In particular, the kernel of this map is the module $\mi_{\dt}$, defined on points by
\[
\mi_n = (x_i - x_j \mid i \neq j)
\]

On the other hand, the map
\[
S^{(\dt-1,1)} \otimes \mathbf{A}\{1\} \rightarrow \mi_\dt \rightarrow 0
\]
is given by
\[
(e_i - e_j) \otimes 1 \mapsto x_i-x_j.
\]
The kernel of this map is the submodule of $S^{(\dt-1,1)} \otimes \mathbf{A}\{1\}$ generated in degree 3 by 
\[
(e_1-e_2) \otimes x_3 + (e_3-e_1) \otimes x_2 + (e_2-e_3) \otimes x_1
\]

In particular, we conclude that all of the homology classes in $H_1(U\F_n(G_n))$ are determined by $H_1(U\F_2(G_3))$. This recovers \cite[Lemma 5.5]{ADK}. A straight forward computation also verifies that for all $d \gg 0$, the $\Z[x_1,\ldots,x_d]$-module
\[
\mathcal{H}_1(G_d)
\]
is generated in grade 2, with relations generated in grades $\leq 3$, verifying the second part of Theorem \ref{bettithm} for the first two Betti numbers.

\end{example}

\end{document}